\documentclass{amsart}
\usepackage{amsthm,amscd,amssymb,verbatim,epsf,amsmath,amsfonts,mathrsfs,graphicx}
\usepackage[colorlinks=true,linkcolor=blue,citecolor=blue]{hyperref}

\begin{document}
\theoremstyle{plain}
\newtheorem{Thm}{Theorem}
\newtheorem{Cor}{Corollary}
\newtheorem{Ex}{Example}
\newtheorem{Con}{Conjecture}
\newtheorem{Main}{Main Theorem}
\newtheorem{Lem}{Lemma}
\newtheorem{Prop}{Proposition}
\newtheorem{Pfm}{Theorem}

\theoremstyle{definition}
\newtheorem{Def}{Definition}
\newtheorem{Note}{Note}

\theoremstyle{remark}
\newtheorem{notation}{Notation}
\renewcommand{\thenotation}{}
\setcounter{tocdepth}{1}
\errorcontextlines=0
\numberwithin{equation}{section}
\renewcommand{\rm}{\normalshape}%

\title[Hopf hypersurfaces]%
   {Hopf hypersurfaces in spaces \\ of oriented geodesics}

\author{Nikos Georgiou}
\address{Nikos Georgiou\\
          Department of Mathematics\\
          Waterford Institute of Technology\\
          Waterford\\
          Co. Waterford\\
          Ireland.}
\email{ngeorgiou@wit.ie}
\author{Brendan Guilfoyle}
\address{Brendan Guilfoyle\\
          School of Science, Technology, Engineering and Mathematics,\\
          Institute of Technology, Tralee\\
          Clash\\
          Tralee\\
          Co. Kerry\\
          Ireland.}
\email{brendan.guilfoyle@ittralee.ie}

\keywords{Hopf hypersurfaces, space of oriented geodesics, hyperbolic $n$-space, $n$-sphere, spaces of constant curvature}
\subjclass{Primary: 53A35; Secondary: 57N13}

\begin{abstract}
A Hopf hypersurface in a (para-)Kaehler manifold is a real hypersurface for which one of the principal directions of the second fundamental form is the 
(para-)complex dual of the normal vector. 

We consider particular Hopf hypersurfaces in the space of oriented geodesics of a non-flat space form of 
dimension greater than 2. For spherical and hyperbolic space forms, the oriented geodesic space admits a canonical Kaehler-Einstein and 
para-Kaehler-Einstein  structure, respectively, so that a natural notion of a Hopf hypersurface exists. 

The particular hypersurfaces considered are formed by the oriented geodesics that are tangent to a given convex hypersurface in the underlying 
space form. We prove that a tangent hypersurface is Hopf in the space of oriented geodesics with respect to this 
canonical (para-)Kaehler structure iff the underlying convex hypersurface is totally umbilic and non-flat.

In the case of 3 dimensional space forms, however, there exists a second canonical complex structure which can also be used to define 
Hopf hypersurfaces. We prove that 
in this dimension, the tangent hypersurface of a convex hypersurface in the space form is always Hopf with respect to this second complex structure.

\end{abstract}

\date{23rd August 2016}

\maketitle
\tableofcontents

\section{Background and Results}

Submanifold theory and in particular the study of real hypersurfaces in a complex manifold, has been of great interest for the last decades  
(for further study see \cite{CM} and \cite{NR}). Let $(M,g,J)$ be a K\"ahler structure, where $M$ is a $2n$-real dimensional manifold, $g$ stands 
for the pseudo-Riemannian metric and $J$ denotes either a complex or paracomplex structure. If $\Sigma$ is a non-degenerate real hypersurface of 
$M$ then there exists a unit normal vector field $N$ along $\Sigma$.  The {\it structure vector field} of $\Sigma$ is the tangential vector field 
$\xi$ given by $\xi:=-J N$. A {\it Hopf hypersurface} is a real hypersurface in a K\"ahler manifold whose structure vector field is a principal direction. 

The principal curvature associated to the structure vector field is called a {\it Hopf principal curvature}. For Riemannian complex space forms, 
Madea in \cite{Ma}, Ki and Suh in \cite{KS}, have proved that the Hopf principal curvature in a Hopf hypersurface must be constant. The same 
statement for pseudo-Riemannian complex space forms and for paracomplex space forms has been proved recently by Anciaux and Panagiotidou in \cite{AK}. 
Furthermore, depending on the size of the Hopf principal curvature, a local characterization of Hopf hypersurfaces is obtained in complex space forms 
\cite{AK} \cite{CR} \cite{Mo}.

The space ${\mathbb L}({\mathbb S}_{\epsilon}^{n+1})$ of oriented geodesics of a real space form ${\mathbb S}_{\epsilon}^{n+1}$ provides a new class 
of (para-) complex manifolds for $n\geq2$. Here ${\mathbb S}_{\epsilon}^{n+1}$ is the round $(n+1)$-sphere ${\mathbb S}^{n+1}$ when $\epsilon=1$ while, 
for $\epsilon=-1$ the real space form ${\mathbb S}_{\epsilon}^{n+1}$ is the hyperbolic $(n+1)$-space ${\mathbb H}^{n+1}$. 

In particular, ${\mathbb L}({\mathbb S}_1^{n+1})$ admits a canonical K\"ahler structure 
$({\mathbb J},{\mathbb G})$, where ${\mathbb J}$ is a complex structure and ${\mathbb L}({\mathbb S}_{-1}^{n+1})$ admits a canonical para-K\"ahler 
structure which will be also denoted by 
$({\mathbb J},{\mathbb G})$. 

In both cases, the metric ${\mathbb G}$ is Einstein and together with ${\mathbb J}$ are both invariant under 
the natural action of the group of isometries of ${\mathbb S}_{\epsilon}^{n+1}$ (see \cite{agk} and \cite{An}). 
The relation between submanifold theory of ${\mathbb S}_{\epsilon}^{n+1}$ and ${\mathbb L}({\mathbb S}_{\epsilon}^{n+1})$ has  been explored by 
several authors recently (see \cite{An} \cite{CU} \cite{gag1} \cite{gag2} and \cite{gk1}). For example, the Gauss map of hypersurfaces in 
${\mathbb S}_{\epsilon}^{n+1}$ correspond to Lagrangian submanifolds in ${\mathbb L}({\mathbb S}_{\epsilon}^{n+1})$. 

The purpose of this paper is to study hypersurfaces in ${\mathbb L}({\mathbb S}_{\epsilon}^{n+1})$ that are formed by the oriented geodesics tangent 
to a submanifold in ${\mathbb S}_{\epsilon}^{n+1}$, called {\it tangent hypersurfaces}. These hypersurfaces were introduced in \cite{gk2} and further 
explored in \cite{gg1}.

In particular, we study  Hopf tangent hypersurfaces in $({\mathbb L}({\mathbb S}_{\epsilon}^{n+1}),{\mathbb J},{\mathbb G})$ and we prove the following:
\begin{Thm}\label{t:secondmaintheorem} 
The tangent hypersurface ${\mathcal H}(\Sigma)$ of an $n$-dimensional submanifold $\Sigma\subset {\mathbb S}^{n+1}$ (resp. hyperbolic space 
${\mathbb H}^{n+1}$) for $n\geq2$ is a Hopf hypersurface of $({\mathbb L}({\mathbb S}^{n+1}),{\mathbb J},{\mathbb G})$ (resp. 
$({\mathbb L}({\mathbb H}^{n+1}),{\mathbb J},{\mathbb G})$) iff it is totally umbilic and non-flat. 
\end{Thm}

\vspace{0.1in}

In 3 dimensions, the space ${\mathbb L}({\mathbb S}_{\epsilon}^3)$ admits a second canonical complex structure, ${\mathbb J}'$, which is also 
invariant under the 
natural action of the group of isometries of ${\mathbb S}_{\epsilon}^3$. Using ${\mathbb J}'$ it is possible to obtain another invariant metric 
$\overline{\mathbb G} $ on ${\mathbb L}({\mathbb S}_{\epsilon}^3)$ (see equation (\ref{e:neutralmetric})).  The metric 
$\overline{\mathbb G}$ is of neutral signature and is locally conformally flat \cite{An}  \cite{Salvai}. 

Tangent hypersurfaces in ${\mathbb L}({\mathbb S}_{\epsilon}^3)$ 
have been studied studied using the neutral metric $\overline{\mathbb G}$ in \cite{gg1}. 
In particular, the tangent hypersurface of an embedded strictly convex 2-sphere is null, i.e., the unit normal vector field has zero length with 
respect to the neutral metric. Furthermore, the totally null planes form a pair of plane fiels on the tangent hypersurface that are contact. 

Regarding the Einstein metric ${\mathbb G}$ we show:

\vspace{0.1in}
\begin{Thm}\label{t:firstmaintheorem} Let $S$ be a smooth closed convex surface in ${\mathbb S}_{\epsilon}^3$. Then the tangent hypersurface ${\mathcal H}(S)$ is a Hopf hypersurface of $({\mathbb L}({\mathbb S}_{\epsilon}^3),{\mathbb J}',{\mathbb G})$. 
\end{Thm}
\vspace{0.1in}

In the next secion we establish notation and preliminaries, while Section \ref{s:3} contains the proof of Theorem \ref{t:secondmaintheorem}. The proof of 
Theorem \ref{t:firstmaintheorem} is in Section \ref{s:4}.

\vspace{0.3in}

\section{Notation and Preliminaries}\label{s:2}

We adopt the notation of section 3.2 of \cite{gg1}, extended to  higher dimensions as in \cite{An}. 

Let ${\mathbb S}_{\epsilon}^{n+1}=\{x\in{\mathbb R}^{n+2}\; :\; \left<x,x\right>_{\epsilon}=1\}$ be the $(n+1)$-(pseudo)-sphere in the Euclidean space 
${\mathbb R}_{\epsilon}^{n+2}:=({\mathbb R}^{n+2},\left<.,.\right>_{\epsilon})$ for $n\geq2$. 
Note that ${\mathbb S}_{1}^{n+1}$ is the round $(n+1)$-sphere ${\mathbb S}^{n+1}$, while 
${\mathbb S}_{-1}^{n+1}$ is anti-isometric to the hyperbolic $(n+1)$-space ${\mathbb H}^{n+1}$. 

The space of oriented geodesics ${\mathbb L}({\mathbb S}_{\epsilon}^{n+1})\subset \Lambda^2({\mathbb R}^{n+2})$ of 
$({\mathbb S}_{\epsilon}^{n+1},g_{\epsilon})$ is $2n$-dimensional and ${\mathbb L}({\mathbb S}_{1}^{n+1})$ can be identified with the 
Grasmannian of oriented planes in ${\mathbb R}_1^{n+2}$, while ${\mathbb L}({\mathbb S}_{-1}^{n+1})$ can be identified with the Grasmannian 
of oriented planes in ${\mathbb R}_{-1}^{n+2}$ such that the induced metric is Lorentzian \cite{An}. 

Recall the complex (resp. paracomplex) structure ${\mathbb J}_{\epsilon}$ on ${\mathbb L}({\mathbb S}_{\epsilon}^{n+1})$ defined by:
\[
{\mathbb J}_{\epsilon}:T_{x\wedge y}{\mathbb L}({\mathbb S}_{\epsilon}^{n+1})\rightarrow :T_{x\wedge y}{\mathbb L}({\mathbb S}_{\epsilon}^{n+1}):x\wedge X+y\wedge Y\mapsto y\wedge X-x\wedge Y,
\]
and simply write ${\mathbb J}$ for ${\mathbb J}_{\epsilon}$. Finally, consider the $SO(n+2)$ (resp. $SO(1,n+1)$)-invariant Einstein metric ${\mathbb G}_{\epsilon}$, given by
\[
{\mathbb G}_{\epsilon}=\iota^{\star}\left<\left<.,.\right>\right>_{\epsilon},
\]
where $\left<\left<.,.\right>\right>_{\epsilon}$ is the flat metric of $\Lambda^2({\mathbb R}^{n+2})$. Then, $({\mathbb L}({\mathbb S}^{n+1}),{\mathbb J},{\mathbb G})$ (resp. $({\mathbb L}({\mathbb H}^{n+1}),{\mathbb J},{\mathbb G})$ ) is a (resp. para-) K\"ahler structure \cite{agk} \cite{An} \cite{gag1}.

The four-dimensional manifold ${\mathbb L}({\mathbb S}_{\epsilon}^3)$ enjoys other natural complex structure,
which is defined as follows: the orthogonal two-plane $(x\wedge y)^{\bot}$ is Riemannian and admits
a canonical orientation (that orientation compatible with the orientations of $x\wedge y$ and ${\mathbb R}^4$). Thus it enjoys a canonical complex structure $J'$. The following endomorphism 
\[
{\mathbb J}'(x\wedge X + y\wedge Y) := x\wedge (J'X) + y\wedge (J'Y ),
\]
defines another complex structure on ${\mathbb L}({\mathbb S}_{\epsilon}^3)$ that is compatible with ${\mathbb G}$. Thus, $({\mathbb L}({\mathbb S}_{\epsilon}^3), {\mathbb G}, {\mathbb J}')$ is another K\"ahler structure (see \cite{agk} \cite{An} \cite{CU}). Since ${\mathbb J}$ and ${\mathbb J}'$ commute, we may define the following metric on ${\mathbb L}({\mathbb S}_{\epsilon}^3)$:
\begin{equation}\label{e:neutralmetric}
\overline{\mathbb G}(\cdot,\cdot)=-\epsilon {\mathbb G}(\cdot,{\mathbb J}\circ {\mathbb J}'\cdot),
\end{equation}
which is of neutral signature and locally conformally flat.
\vspace{0.1in}

\begin{Def}
A {\it tangent hypersurface} ${\mathcal H}(\Sigma)$ over a hypersurface $\Sigma$ in ${\mathbb S}_{\epsilon}^{n+1}$ is the hypersurface  of ${\mathbb L}({\mathbb S}_{\epsilon}^{n+1})$ formed by the oriented geodesics in ${\mathbb S}_{\epsilon}^{n+1}$ tangent to $\Sigma$ at some point.
\end{Def}
\vspace{0.1in}

This was introduced for $n=2$ in the flat case in \cite{gk2} and the curved case in \cite{gg1}. In this dimension ${\mathcal H}(S)$ is 
$\overline{\mathbb G}$-null, i.e., the unit normal vector field is 
of zero length  with respect to the metric $\overline{\mathbb G}$. Furthermore, ${\mathcal H}(S)$ is locally a circle bundle over $S$, with projection 
$\pi: {\mathcal H}(S)\rightarrow  S$ and fibre generated by rotation about the normal to $S$. For further details and properties in this dimension, 
see \cite{gg1}.

\vspace{0.2in}

\section{Hopf Tangent Hypersurfaces}\label{s:3}

In this section we consider the conditions under which a tangent hypersurface is Hopf with respect to the canonical (para-)Kaehler
structure $({\mathbb J},{\mathbb G})$. 

We start with the following Lemma:

\begin{Lem}\label{l:lem}
Let $(e_1,\ldots,e_n)$ be an orthonormal basis of ${\mathbb R}^n$. Then, for every $v\in {\mathbb S}^{n-1}\subset {\mathbb R}^n$ there exist $\theta_1\in [0,2\pi)$ and $\theta_2,\ldots,\theta_{n-1}\in [-\pi/2,\pi/2]$, such that
\[
v=\cos\theta_1\ldots\cos\theta_{n-1}\,e_1+\sin\theta_1\cos\theta_2\ldots\cos\theta_{n-1}\,e_2+\sin\theta_2\cos\theta_3\ldots\cos\theta_{n-1}\,e_3+
\]
\[
\qquad\qquad\ldots+\sin\theta_{n-2}\cos\theta_{n-1}\,e_{n-1}+\sin\theta_{n-1}\,e_n.
\]
\end{Lem}
\begin{proof}
Since $\left<e_i,e_j\right>=\delta_{ij}$, every vector $v$ in ${\mathbb R}^n$ satisfies
\[
\left<v,v\right>=\left<v,e_1\right>\left<v,e_1\right>+\ldots+\left<v,e_n\right>\left<v,e_n\right>,
\]
and the fact that $v\in {\mathbb S}^{n-1}$ yields,
\begin{equation}\label{e:component}
\left<v,e_1\right>^2+\ldots+\left<v,e_n\right>^2=1.
\end{equation}
Then, 
\[
|\left<v,e_n\right>|\leq 1,
\]
Thus, there exists $\theta_{n-1}\in [-\pi/2,\pi/2]$ such that 
\begin{equation}\label{e:lastcomponent}
\left<v,e_n\right>=\sin\theta_{n-1}.
\end{equation}
Using (\ref{e:lastcomponent}), we get,
\begin{equation}\label{e:lastcomponent21}
\left<v,e_1\right>^2+\ldots+\left<v,e_{n-1}\right>^2=\cos^2\theta_{n-1}.
\end{equation}
If $|\theta_{n-1}|=\pi/2$, we have
\[
\left<v,e_1\right>=\ldots=\left<v,e_{n-1}\right>=0,
\]
and choosing $\theta_1=\ldots=\theta_{n-2}=0$, we obtain $v=e_n$. Similar argument shows that if $|\theta_k|=\pi/2$ for some $k$, then $\theta_i=0$ for all $i<k$.

Suppose that $|\theta_k|\neq\pi/2$ for all $k$. Following (\ref{e:lastcomponent21}) we have
\[
\left(\frac{\left<v,e_1\right>}{\cos\theta_{n-1}}\right)^2+\ldots+\left(\frac{\left<v,e_{n-1}\right>}{\cos\theta_{n-1}}\right)^2=1.
\]
We then have 
\[
\left|\frac{\left<v,e_{n-1}\right>}{\cos\theta_{n-1}}\right|\leq 1
\]
and so there exists $\theta_{n-2}\in [-\pi/2,\pi/2]$ such that 
\[
\frac{\left<v,e_{n-1}\right>}{\cos\theta_{n-1}}=\sin\theta_{n-2}.
\]
It follows,
\begin{equation}\label{e:slastcomponent}
\left<v,e_{n-1}\right>=\sin\theta_{n-2}\cos\theta_{n-1}.
\end{equation}
From (\ref{e:slastcomponent}), we obtain
\[
\left(\frac{\left<v,e_1\right>}{\cos\theta_{n-1}}\right)^2+\ldots+\left(\frac{\left<v,e_{n-2}\right>}{\cos\theta_{n-1}}\right)^2=\cos^2\theta_{n-2},
\]
which yields,
\[
\left(\frac{\left<v,e_1\right>}{\cos\theta_{n-2}\cos\theta_{n-1}}\right)^2+\ldots+\left(\frac{\left<v,e_{n-2}\right>}{\cos\theta_{n-2}\cos\theta_{n-1}}\right)^2=1,
\]
and hence there exists $\theta_{n-3}\in [-\pi/2,\pi/2]$ such that 
\[
\frac{\left<v,e_{n-1}\right>}{\cos\theta_{n-2}\cos\theta_{n-1}}=\sin\theta_{n-3}.
\]
Equivalently,
\[
\left<v,e_{n-2}\right>=\sin\theta_{n-3}\cos\theta_{n-2}\cos\theta_{n-1}.
\]
Applying the same process we obtain angles $\theta_2,\ldots,\theta_{n-1}\in[-\pi/2,\pi/2]$, satisfying 
\[
\left<v,e_k\right>=\sin\theta_{k-1}\cos\theta_k\ldots\cos\theta_{n-1},\quad k=3,\ldots, n.
\]
We then have,
\[
\left(\frac{\left<v,e_1\right>}{\cos\theta_2\ldots\cos\theta_{n-1}}\right)^2+\left(\frac{\left<v,e_2\right>}{\cos\theta_2\ldots\cos\theta_{n-1}}\right)^2=1.
\]
Thus, there exists $\theta\in[0,2\pi)$, such that 
\[
\frac{\left<v,e_1\right>}{\cos\theta_2\ldots\cos\theta_{n-1}}=\cos\theta_1\qquad\qquad \frac{\left<v,e_2\right>}{\cos\theta_2\ldots\cos\theta_{n-1}}=\sin\theta_1,
\]
and the lemma follows.
\end{proof}

\begin{Def}Let $(M,g)$ be a smooth manifold and $\Sigma$ be a hypersurface in $M$. A point $x\in\Sigma$ is said to be {\it umbilic} if the second fundamental form $h$ is proportional to the first fundamental form, i.e. there exists a constant $\lambda$ such that 
\[
h(X,Y)=\lambda g(X,Y).
\]
A hypersurface is said to be {\it totally umbilic} if all its points are umbilic. In particular, for every point in a totally umbilic hypersurface all principal curvatures are equal.
\end{Def}

\vspace{0.1in}

\noindent {\it Proof of Theorem} \ref{t:secondmaintheorem}:
Any vector field $X$ in $\Sigma$ is identified with $d\phi(X)$ and let $e_1,\ldots,e_n$ be the principal directions of $\phi$ with corresponding principal curvatures $\lambda_1,\ldots,\lambda_n$.
Using Lemma \ref{l:lem}, the tangent hypersurface ${\mathcal H}(\Sigma)$ can be locally parametrized by $$\bar\phi:\Sigma\times{\mathbb S}^{n-1}\rightarrow {\mathbb L}({\mathbb S}^{n+1}):(x,\theta_1,\ldots,\theta_{n-1})\mapsto\phi(x)\wedge v(x,\theta_1,\ldots,\theta_{n-1}),$$ where,
\[
v=\cos\theta_1\ldots\cos\theta_{n-1}e_1(x)+\sin\theta_1\cos\theta_2\ldots\cos\theta_{n-1}e_2(x)+
\sin\theta_2\cos\theta_3\ldots\cos\theta_{n-1}e_3(x)
\]
\[
\qquad\qquad+\ldots+\sin\theta_{n-2}\cos\theta_{n-1}\,e_{n-1}(x)+\sin\theta_{n-1}\,e_n(x)
\]
For $k=1,\ldots,n-1$ define,
\[
v_k=\frac{\partial_{\theta_k}v}{|\partial_{\theta_k}v|}.
\]
Then, 
\[
v_k=-\cos\theta_1\ldots\cos\theta_{k-1}\sin\theta_k\,e_1-\sin\theta_1\cos\theta_2\ldots\cos\theta_{k-1}\sin\theta_k\,e_2-
\]
\[
\qquad\qquad-\sin\theta_2\cos\theta_3\ldots\cos\theta_{k-1}\sin\theta_k\,e_3-\sin\theta_3\cos\theta_4\ldots\cos\theta_{k-1}\sin\theta_k\,e_4-
\]
\[
\qquad\qquad-\sin\theta_{k-2}\cos\theta_{k-1}\sin\theta_k\,e_{k-1}-\sin\theta_{k-1}\sin\theta_k\,e_k+\cos\theta_k\,e_{k+1}.
\]
Setting $v_n:=v$, one can show that $\left<v_i,v_j\right>=\delta_{ij}$.

The tangent space $T_{\phi\wedge v}{\mathbb L}({\mathbb S}_{\epsilon}^{n+1})$ on the oriented plane $\phi\wedge v$ in ${\mathbb R}^{n+2}$ is identified with the space of the vector fields that are of the form 
\[
\phi\wedge X+v\wedge Y,
\]
where $X,Y\in (\phi\wedge v)^{\perp}=\mbox{span}\{N,v_1,\ldots,v_{n-1}\}$. Using the (para-) complex structure $J$ defined by $J\phi=v$ and $Jv=-\epsilon\phi$, the (para-) complex structure ${\mathbb J}$ on ${\mathbb L}({\mathbb S}_{\epsilon}^{n+1})$ is defined as follows,
\[
{\mathbb J}(\phi\wedge X+v\wedge Y)=(J\phi)\wedge X+(Jv)\wedge Y=-\epsilon\phi\wedge Y+v\wedge X.
\]
Consider the matrix $(g_{ij})\in SO(n)$, given by $v_k=\displaystyle\sum_{l=1}^n g_{kl}e_l$ and denote the inverse matrix by $(g^{ij})$. It then follows, 
\begin{eqnarray}
d\bar\phi(e_k)&=&d(\phi\wedge v)(e_k)\nonumber\\
&=&e_k\wedge v+\phi\wedge \overline\nabla_{e_k}v\nonumber\\
&=&\sum_{l=1}^{n-1}g^{kl}v_l\wedge v+\phi\wedge \overline\nabla_{e_k}v.\nonumber
\end{eqnarray}
A brief computations gives,
\[
\overline\nabla_{e_k}v=\sum_{l=1}^{n-1}\sum_{s=1}^{n}g^{ks}\left< \overline\nabla_{v_s}v,v_l\right>v_l+\left< \overline\nabla_{e_k}v,\phi\right>\phi+\lambda_kg_{nk} N.
\]
Therefore, the tangent bundle $T{\mathcal H}(\Sigma)$ is generated by the vector fields,
\begin{equation}\label{derivativegeneral}
d\bar\phi(e_k)=\sum_{l=1}^{n-1}g^{kl}v_l\wedge v+\sum_{l=1}^{n-1}\sum_{s=1}^{n}g^{ks}\left< \overline\nabla_{v_s}v,v_l\right>\phi\wedge v_l+\lambda_kg^{kn}\phi\wedge N.
\end{equation}
The unit normal vector field $\bar N$ of ${\mathcal H}(\Sigma)$ in ${\mathbb L}({\mathbb S}^{n+1})$ is given by,
\[
\bar N=v\wedge N.
\]
The structure vector field $\xi=-{\mathbb J}\bar N$ is,
\[
\xi=\phi\wedge N.
\]
Let $D,\overline{D}$ be the Levi-Civita connection of $\left<\left<.,.\right>\right>$ and ${\mathbb G}$, respectively. Then,
\begin{eqnarray}
\overline{D}_{d\bar\phi(e_k)}\bar N&=&\overline{D}_{d\bar\phi(e_k)}(\phi\wedge N)\nonumber \\
&=&\sum_{l=1}^{n-1}\left<\overline{\nabla}_{e_k}v,v_l\right>v_l\wedge N+\left<\overline{\nabla}_{e_k}v,\phi\right>\phi\wedge N+\lambda_k\sum_{l=1}^{n-1}g^{kl}v_l\wedge v.\nonumber
\end{eqnarray}
Thus,
\begin{equation}\label{e:shapeop0}
D_{d\bar\phi(e_k)}\bar N=-g_{nk}\phi\wedge N+\lambda_k\sum_{l=1}^{n-1}g^{kl}v_l\wedge v.
\end{equation}
Similarly,
\[
\overline{D}_{d\bar\phi(\partial/\partial\theta_k)}\bar N=\overline{D}_{d\bar\phi(\partial/\partial\theta_k)}(v\wedge N)=(\partial_{\theta_k} v)\wedge N,
\]
which gives,
\begin{equation}\label{e:shapeop}
D_{d\bar\phi(\partial/\partial\theta_k)}\bar N=0.
\end{equation}
If $A$ stands for the shape operator of ${\mathcal H}(\Sigma)$ in ${\mathbb L}({\mathbb S}^{n+1})$, the relations (\ref{e:shapeop0}) and (\ref{e:shapeop}) give,
\begin{eqnarray}
A(d\bar\phi(e_k))&=&-g_{nk}\phi\wedge N+\lambda_k\sum_{l=1}^{n-1}g^{kl}v_l\wedge v\label{e:shap0}\\
A(d\bar\phi(\partial/\partial\theta_k))&=&0.\label{e:shap1}
\end{eqnarray}
Suppose that all principal curvatures $\lambda_1,\ldots,\lambda_n$ are all equal to $\lambda$, where $\lambda(x)\neq 0$ for all $x\in\Sigma$. Using (\ref{derivativegeneral}) and the fact that  we have, 
\begin{eqnarray}
\sum_{k=1}^ng_{nk}d\bar\phi(e_k)&=&\sum_{k=1}^n\sum_{l=1}^{n-1}g_{nk}g^{kl}v_l\wedge v+\sum_{k=1}^n\sum_{l=1}^{n-1}\sum_{s=1}^{n}g_{nk}g^{ks}\left< \overline\nabla_{v_s}v,v_l\right>\phi\wedge v_l\nonumber \\
&&\qquad\qquad\qquad+\sum_{k=1}^n\lambda_kg_{nk}g^{kn}\phi\wedge N\nonumber \\
&=&\sum_{l=1}^{n-1}\left< \overline\nabla_{v}v,v_l\right>\phi\wedge v_l+\sum_{k=1}^n\lambda g_{nk}g^{kn}\phi\wedge N\nonumber \\
&=&\sum_{l=1}^{n-1}\left< \overline\nabla_{v}v,v_l\right>\phi\wedge v_l+\lambda \xi.\nonumber 
\end{eqnarray}
The expression,
\[
\phi\wedge v_k=\frac{d\phi(\partial/\partial\theta_k)}{|\partial_{\theta_k}v|},
\]
gives,
\[
\sum_{k=1}^ng_{nk}d\bar\phi(e_k)=\sum_{k=1}^{n-1}\frac{\left< \overline\nabla_{v}v,v_k\right>}{|\partial_{\theta_k}v|}
d\bar\phi (\partial/\partial\theta_k)+\lambda\xi.
\]
Hence,
\[
\xi=\lambda^{-1}\sum_{k=1}^n\left(g_{nk}d\bar\phi(e_k)-\frac{\left< \overline\nabla_{v}v,v_k\right>}{|\partial_{\theta_k}v|}
d\bar\phi (\partial/\partial\theta_k)\right).
\]
Using (\ref{e:shap0}) and (\ref{e:shap1}), we finally get
\begin{eqnarray}
A\xi&=&\lambda^{-1}\sum_{k=1}^n\left(g_{nk}A(d\bar\phi(e_k))-\frac{\left< \overline\nabla_{v}v,v_k\right>}{|\partial_{\theta_k}v|}A(d\bar\phi (\partial/\partial\theta_k))\right)\nonumber\\
&=&\lambda^{-1}\sum_{k=1}^n\left(-g_{nk}g_{nk}\phi\wedge N+\lambda_k\sum_{l=1}^{n-1}g_{nk}g^{kl}v_l\wedge v\right)\nonumber \\
&=&-\lambda^{-1}\left(\sum_{k=1}^ng_{nk}^2\right)\xi\nonumber \\
&=&-\lambda^{-1}\xi,\nonumber 
\end{eqnarray}
which shows that ${\mathcal H}(\Sigma)$ is a Hopf hypersurface.

Suppose that ${\mathcal H}(\Sigma)$ is Hopf with respect to $({\mathbb G}, {\mathbb J})$. Assuming that $\phi$ is not totally umbilic, consider the case where the principal curvatures $\lambda_k$ are all equal to $\lambda$ except $\lambda_{k_0}\neq\lambda$. A brief computation gives,
\[
\left(\sum_{k=1}^n\lambda_k g_{nk}g^{kn}\right)A\xi=\xi+(\lambda_s-\lambda)g_{ns}\left(\sum_{l=1}^{n-1}g^{sl}v_l\right)\wedge v,
\]
and shows that ${\mathcal H}(\Sigma)$ is not Hopf. Similar arguments can be used for the cases where two or more principal curvatures differ and the 
Theorem follows. $\square$

\vspace{0.2in}

\section{The Special Case of Dimension 3}\label{s:4}

As mentioned in the introduction, 3 dimensional non-flat space forms are unusual in that their exists a second complex structure ${\mathbb J}'$ on 
the space of oriented geodesics. In this section we consider the conditions under which a tangent hypersurface is Hopf with respect to the Kaehler
structure $({\mathbb J}',{\mathbb G})$. 

Using the terminology introduced in Section \ref{s:2} for dimension 3, we now prove Theorem  \ref{t:firstmaintheorem}:

\vspace{0.1in}

\noindent {\it Proof of Theorem} \ref{t:firstmaintheorem}:
Let $\phi:S\rightarrow {\mathbb S}_{\epsilon}^3$ be an embedding of a closed convex surface $S$ in ${\mathbb S}_{\epsilon}^3$ and let $(e_1,e_2)$ be the principal directions with corresponding pricipal curvatures $\lambda_1,\lambda_2$. Let $N$ be the unit normal vector field along the surface $\phi(S)$ such that $(\phi,e_1,e_2,N)$ is an oriented orthonormal frame of ${\mathbb R}^4$. For $\theta\in {\mathbb S}^1$, define the following tangential vector fields
\[
v(x,\theta)=\cos\theta\,e_1(x)+\sin\theta\, e_2(x)\quad\mbox{and}\quad v^{\bot}(x,\theta)=-\sin\theta\,e_1(x)+\cos\theta\, e_2(x)
\]
The tangent hypersurface ${\mathcal H}(S)$ over $S$  is locally parametrised by
\begin{eqnarray}
\bar\phi:S\times {\mathbb S}^1&\rightarrow& {\mathbb L}({\mathbb S}_{\epsilon}^3)\nonumber\\
(x,\theta)&\mapsto& \phi(x)\wedge v(x,\theta)\nonumber
\end{eqnarray}
Let $\xi'$ be the structure vector field of ${\mathcal H}(S)$ with respect to 
$({\mathbb J}',{\mathbb G})$, that is, $\xi'=-{\mathbb J}'\bar N$. 

Considering the principal directions $(e_1,e_2)$ with pricipal curvatures $\lambda_1,\lambda_2$, the derivative of $\bar\phi$ is given by:
\begin{eqnarray}
d\bar\phi(e_1)&=&v_1\,\phi\wedge v^{\bot}+\lambda_1\cos\theta\,\phi\wedge N+\sin\theta\, v\wedge v^{\bot}\nonumber \\
d\bar\phi(e_2)&=&v_2\,\phi\wedge v^{\bot}+\lambda_2\sin\theta\,\phi\wedge N-\cos\theta\, v\wedge v^{\bot}\nonumber \\
d\bar\phi(\partial/\partial\theta)&=&\phi\wedge v^{\bot},\nonumber 
\end{eqnarray}
for some smooth functions $v_1$ and $v_2$. Clearly, ${\mathcal H}$ is non-degenerate, with respect to ${\mathbb G}$, and the orthonormal normal vector field $\bar N$ is given by
\[
\bar N=v\wedge N.
\]
Let $\bar D, D$ be the Levi-Civita connections of $\left<\left<.,.\right>\right>_{\epsilon}$ and ${\mathbb G}$, respectively. Denote by $A$ and $h$ the shape operator and the second fundamental form of $\bar\phi$ and let $\bar h$ be the second fundamental form of the inclusion map
$\iota:{\mathbb L}({\mathbb S}_{\epsilon}^3)\hookrightarrow \Lambda^2({\mathbb R}^4)$. Note that for any vector fields $X,Y$ of ${\mathcal H}(S)$, we have:
\[
{\mathbb G}(h(X,Y),\bar N)={\mathbb G}(AX,Y).
\]
It follows,
\begin{eqnarray}
-\bar D_{d\bar\phi(e_1)}\bar N
&=& -v_1\,v^{\bot}\wedge N +\cos\theta\,\phi\wedge N-\lambda_1\sin\theta\, v\wedge v^{\bot}\nonumber 
\end{eqnarray}
Now, 
\begin{eqnarray}
A(d\bar\phi(e_1))
&=&-\bar D_{d\bar\phi(e_1)}\bar N+\bar h(d\bar\phi(e_1),\bar N),\nonumber 
\end{eqnarray}
which yields,
\begin{equation}\label{e:shapefore1}
A(d\bar\phi(e_1))= \cos\theta\,\phi\wedge N-\lambda_1\sin\theta\, v\wedge v^{\bot}.
\end{equation}
Similarly we get,
\begin{equation}\label{e:shapefore2}
A(d\bar\phi(e_2))= \sin\theta\,\phi\wedge N+\lambda_2\cos\theta\, v\wedge v^{\bot}\qquad A(d\bar\phi(\partial/\partial\theta))= 0.
\end{equation}
Using (\ref{e:shapefore1}), we have
\begin{eqnarray}
{\mathbb G}(h(d\bar\phi(e_1),d\bar\phi(e_1)),\bar N)
&=& \lambda_1\cos 2\theta.\nonumber 
\end{eqnarray}
Analogously we have,
\[
{\mathbb G}(h(d\bar\phi(e_1),d\bar\phi(e_2)),\bar N)={\mathbb G}(h(d\bar\phi(e_2),d\bar\phi(e_1)),\bar N)= H\sin 2\theta,
\]
\[
{\mathbb G}(h(d\bar\phi(e_2),d\bar\phi(e_2)),\bar N)=- \lambda_2\cos 2\theta.
\]
and 
\[
{\mathbb G}(h(d\bar\phi(e_1),d\bar\phi(\partial/\partial\theta)),\bar N)={\mathbb G}(h(d\bar\phi(e_2),d\bar\phi(\partial/\partial\theta)),\bar N)= 0,
\]
\[
{\mathbb G}(h(d\bar\phi(\partial/\partial\theta,d\bar\phi(\partial/\partial\theta)),\bar N)= 0.
\]
 In terms of $(e_0:=d\bar\phi(\partial/\partial\theta), d\bar\phi(e_1),d\bar\phi(e_2))$, the second fundamental form $h$ can be expressed by the following symmetric matrix
\[
h=\begin{pmatrix}0 & 0 & 0 \\ 0 & \lambda_1\cos 2\theta & H\sin 2\theta \\ 0 & H \sin 2\theta & -\lambda_2\cos 2\theta\end{pmatrix}
\]
The principal curvatures are the eigenvalues of $h$, which are $0, \lambda_+$ and $\lambda_-$, where
\[
\lambda_{+}=\lambda_1\cos^2\theta+\lambda_2\sin^2\theta\qquad \lambda_{-}=-\lambda_1\sin^2\theta-\lambda_2\cos^2\theta,
\]
with corresponding principal directions $e_0,v_+$ and $v_-$. Then,
\begin{eqnarray}
v_+&=& \cos\theta\,d\bar\phi(e_1)+\sin\theta\,d\bar\phi(e_2),\nonumber
\end{eqnarray}
and thus,
\[
v_+=\left<\nabla_v v,v^{\perp}\right>\,\phi\wedge v^{\perp}+(\lambda_1\cos^2\theta+\lambda_2\sin^2\theta)\;\phi\wedge N.
\]
The fact that $S$ is closed and convex implies that
\[
\lambda_+\lambda_-<0,
\]
and $\{v_+,e_0\}$ are linearly independent. Thus, the principal directions $e_0$ and $v_+$ span the $\alpha$-plane $\Pi_+$ \cite{gg1}, that is,
\[
\Pi_+=\mbox{span}\{e_0,v_+\}.
\]
It can be easily proved that
\begin{equation}\label{e:relationofnormals}
{\mathbb J}e_0={\mathbb J}'\bar N=-\xi'.
\end{equation}
Since ${\mathbb J}\Pi_+=\Pi_+$, it then follows that $\xi'\in \Pi_+$ and thus $\xi'$ is a principal direction.  Hence, ${\mathcal H}(S)$ is a Hopf hypersurface of $({\mathbb L}({\mathbb S}^3_\epsilon), {\mathbb J}',{\mathbb G})$. $\square$

\end{document}